\documentclass[12pt,a4paper,oneside]{amsart}
\pdfoutput=1
\usepackage{amssymb}
\usepackage{esint}
\usepackage{mathtools}
\usepackage{hyperref}
\mathtoolsset{showonlyrefs}
\usepackage[english]{babel}
\usepackage[latin1]{inputenc}
\usepackage[pdftex]{graphicx}
\theoremstyle{plain}
\newtheorem{theorem}{Theorem}
\newtheorem{lemma}[theorem]{Lemma}
\newtheorem{corollary}[theorem]{Corollary}
\newtheorem{proposition}[theorem]{Proposition}

\theoremstyle{definition}

\theoremstyle{remark}
\newtheorem{remark}[theorem]{Remark}
\newcommand{\R}{\mathbb R}
\newcommand{\C}{\mathbb C}
\newcommand{\N}{\mathbb N}
\newcommand{\Nb}{\bar{\mathbb N}}

\newcommand{\h}{\mathcal H}

\newcommand{\brt}[1]{{\mathcal G_{#1}}}

\newcommand{\der}{\mathrm{d}}
\newcommand{\eps}{\varepsilon}
\renewcommand{\phi}{\varphi}
\newcommand{\abs}[1]{\left| #1 \right|}

\newcommand{\sisus}{\operatorname{int}}

\renewcommand{\theta}{\vartheta}

\newcommand{\Order}{\mathcal O}
\newcommand{\order}{o}

\newcommand{\lap}{\mathcal L}
\newcommand{\sff}[2]{\mathrm{I\!I}\!\left(#1,#2\right)}
\newcommand{\tpm}[1]{\sff{#1}{#1}}
\newcommand{\ip}[2]{\left\langle#1,#2\right\rangle}
\newcommand{\bd}[1]{{\check{#1}}}
\newcommand{\dbd}[1]{\dot{\bd{#1}}}
\newcommand{\ddbd}[1]{\ddot{\bd{#1}}}
\newcommand{\Der}[1]{\frac{\der}{\der #1}}
\newcommand{\set}[1]{\underline{#1}}
\newcommand{\E}{{\mathcal E}}
\title{Boundary reconstruction for the broken ray transform}
\author{Joonas Ilmavirta}
\address{Department of Mathematics and Statistics, University of Jyv\"askyl\"a, P.O. Box 35 (MaD) FI-40014 University of Jyv\"askyl\"a, Finland}
\email{joonas.ilmavirta@jyu.fi}
\date{\today}

\begin{document}

\begin{abstract}
We reduce boundary determination of an unknown function and its normal derivatives from the (possibly weighted and attenuated) broken ray data to the injectivity of certain geodesic ray transforms on the boundary.
For determination of the values of the function itself we obtain the usual geodesic ray transform, but for derivatives this transform has to be weighted by powers of the second fundamental form.
The problem studied here is related to Calder\'on's problem with partial data.
\end{abstract}

\subjclass[2010]{
Primary
53C65, 
78A05; 
Secondary
35R30, 
58J32
}

\keywords{Broken ray transform, X-ray transform, Calder\'on's problem, inverse problems}

\maketitle

\section{Introduction}

For a compact Riemannian manifold~$M$ with boundary, we study the problem of recovering an unknown function $f:M\to\R$ and its derivatives at the boundary from its broken ray transform.
We turn this boundary determination problem to injectivity of weighted geodesic ray transforms (X-ray transforms) on the boundary manifold~$\partial M$.
If these transforms are injective and~$f$ is smooth enough, one may recover the Taylor series of~$f$ at any point on~$\partial M$.
This result is stated in theorem~\ref{thm:f-rec} and corollary~\ref{cor:f-rec}.

A broken ray is a curve in~$M$ which is geodesic in $\sisus M$ and reflects by the usual reflection law (angle of incidence equals angle of reflection) at~$\partial M$.
We fix a set $E\subset\partial M$, called the set of tomography, where measurements can be done.
The broken ray transform of a function $f:M\to\R$ is a function that takes a broken ray with both endpoints in~$E$ into the integral of~$f$ along the broken ray.
We allow a broken ray to reflect also on~$E$ if that is convenient.
Weight and attenuation may be included in the broken ray transform.
The broken ray transform in this setting has been studied in~\cite{KS:calderon,I:disk,H:square,I:refl}.

The results presented here are based on the observation that broken rays in a suitably convex region of the boundary can tend to a geodesic on the boundary.
This observation allows reconstruction of the boundary values under favorable circumstances, though recovery of normal derivatives requires more careful analysis of the details of this convergence.
This phenomenon is well known in the study of billiards, and billiard trajectories (broken rays) near or at the boundary are known as `glancing billiards'~\cite{M:ergodic-dynamical}, `creeping rays'~\cite{Z:spectral-survey}, `whispering gallery trajectories'~\cite{Z:billiard-boundary}, or `gliding rays'~\cite{AM:gliding-rays}.

\subsection{Main results}

These results assume that the broken ray transform of $f\in C^\infty(M)$ with respect to a given set of tomography $E\subset\partial M$ is known.

Our main result is theorem~\ref{thm:f-rec} and it states the following:
Suppose~$\sigma$ is a geodesic on~$\partial M$ with endpoints in $\sisus E$ along which~$\partial M$ is strictly convex.
Then the integral of~$f$ along~$\sigma$ may be constructed.
Furthermore, if the normal derivatives of~$f$ of orders $0,1,\dots,k-1$ are known in a neighborhood of~$\sigma$, one may recover the integral of~$\partial_\nu^kf$ along~$\sigma$ weighted by the second fundamental form to the power~$-k/3$.

Corollary~\ref{cor:f-rec} follows easily:
Suppose~$\partial M$ is strictly convex.
If the geodesic ray transform on $\partial M\setminus E$ weighted with any power~$-k/3$, $k\in\N$, is injective, then one may recover the Taylor polynomial of~$f$ at every boundary point.

The theorem with its corollary remains true if one introduces a weight and attenuation in the broken ray transform.
These results are stated in more detail in section~\ref{sec:thm}.

\subsection{Earlier and related results}

Eskin~\cite{eskin} reduced a partial data problem for the electromagnetic Schr\"odinger operator to injectivity of the broken ray transform and proved this injectivity under some geometrical conditions.
Kenig and Salo~\cite{KS:calderon} recently showed that partial data for Calder\'on's problem (as reduced to the Schr\"odinger equation) is enough to reconstruct conductivity in a certain tubular manifold if the broken ray transform is injective on the transversal manifold.
The result~\cite[Theorem~2.4]{KS:calderon} motivates the study of the broken ray transform for general Riemannian manifolds.


The broken ray transform has gained interest recently, and injectivity results have been given in the disc~\cite{I:disk} (partial result), square~\cite{H:square}, and conical sets~\cite{I:refl}.
The aim of the present paper is to provide tools for reconstructing the boundary values of a function from its broken ray transform.
The reason for focusing attention to the boundary is twofold: first, the problem reduces to ray transforms without reflections, which are easier to analyze, and second, boundary reconstruction can be a useful first step in interior determination.

Boundary determination results for a number of inverse problems have been obtained, and they have been used for interior determination.
In particular, for Calder\'on's problem, boundary determination (see e.g.~\cite{KV:bdy-det-calderon,SU:calderon,B:calderon-boundary,NT:calderon-boundary,SZ:p-calderon}) allows one to convert the conductivity equation into a Schr\"odinger equation, for which the problem is easier to study.
For a review of Calder\'on's problem we refer to~\cite{U:eit-calderon}.

For boundary determination of a metric from the boundary distance function, see~\cite{Z:bdy-det-jet} and references therein.
Boundary rigidity is related to ray transforms via linearization.

There are also boundary determination results for ray transforms.
Recovering the boundary value of a scalar function from its ray transform on a strictly convex manifold is trivial, but recovering tensors or derivatives requires more work.
As an example of a boundary determination result we mention~\cite[Lemma~2.1]{S:int-geom-nonconvex}, which is used as a step towards interior determination.
This boundary determination result was later generalized by Stefanov and Uhlmann~\cite{SU:bdy-rig-simple}.

Whereas a boundary determination result only gives the unknown function and possibly its derivatives at the boundary, a local support theorem gives it in a neighborhood of the boundary.
Such a local support theorem in the Euclidean case follows from Helgason's global support theorem~\cite[Theorem~2.6]{book-helgason}, and there are also more recent results on manifolds~\cite{K:spt-thm,UV:local-x-ray}.
Boundary recovery for ray transforms at the accessible part of the boundary is rather different in nature from the present problem.
In general, boundary recovery results are local, but the reconstruction scheme presented here is not local on the reflecting part of the boundary.
Our method is local only in the sense that one only needs to consider broken rays in an arbitrary neighborhood of the boundary.

The ray transform problem to which which we transform the boundary reconstruction problem includes weights.
Attenuated ray transforms have been studied extensively, but the case with general weights has received less attention -- possibly because it does not arise as often in other problems.
There are, however, some recent results in this direction~\cite{FSU:general-x-ray,SUV:partial-bdy-rig}.
There are also counterexamples in the Euclidean plane given by Boman~\cite{B:weighted-radon-noninjective}.

\subsection{Outline}

We organize our paper as follows.
In section~\ref{sec:ass} we present notation and assumptions that we will use throughout.
We present our main result and some immediate corollaries in section~\ref{sec:thm}.
We prove the necessary geometrical lemmas in section~\ref{sec:lma} and use them to prove the main result in section~\ref{sec:pf}.
Finally we consider the example of boundary determination for surfaces in section~\ref{sec:1d}.

\section{Notation and precise statement of results}

\subsection{Notation and assumptions}
\label{sec:ass}

The assumptions given in this section are assumed to hold throughout the subsequent discussion without mention.

The setting of the problem involves a compact Riemannian manifold~$M$ with boundary~$\partial M$ and a~$C^3$ metric~$g$, and \emph{the set of tomography}~$E\subset\partial M$.
We do not assume that~$E$ is open.
We write $m=\dim M$ and assume~$m\geq2$.

We assume that the boundary is~$C^3$; this ensures that the (local) change to boundary normal coordinates is a~$C^2$ diffeomorphism.
Assuming that the boundary is~$C^3$ will be important -- see remark~\ref{rmk:C2} for details.


A broken ray is a geodesic in~$M$ which reflects on~$\partial M$ according to the usual reflection law: the angle of incidence equals the angle of reflection.
All broken rays are assumed to have finite length and unit speed.

We assume the attenuation coefficient~$a$ and weight~$w$ to be known continuous functions on~$TM$.
For brevity, we denote
\begin{equation}
\label{eq:W-def}
W_\alpha(t)=w(\alpha(t),\dot{\alpha}(t))\exp\left(\int_0^t a(\alpha(s),\dot{\alpha}(s))\der s\right)
\end{equation}
for any piecewise~$C^1$ curve $\alpha:[0,L]\to M$.
(The integral is evaluated piecewise if~$\alpha$ is not~$C^1$.)
If there is no weight or attenuation, then $a\equiv0$, $w\equiv1$ and $W\equiv1$.

We denote by $\Gamma_E$ the set of broken rays with both endpoints on~$E$.
We define the broken ray transform of a continuous function~$f:M\to\R$ as $\brt{W}f:\Gamma_E\to\R$ by letting
\begin{equation}
\label{eq:brt-def}
\brt{W}f(\gamma)=\int_0^LW_\gamma(t)f(\gamma(t))\der t
\end{equation}
for $\gamma:[0,L]\to M$.

We use boundary normal coordinates near~$\partial M$; they can be used up to some distance~$h>0$ from the boundary.
The normal direction corresponds to the zeroth coordinate and the inward unit normal vector is therefore $\nu^\alpha=\delta^\alpha_0$ in these coordinates, where~$\delta$ is the Kronecker symbol.
Greek indices $\alpha,\beta,\dots$ take values $0,1,2,\dots$ and Latin indices $i,j,\dots$ take values $1,2,\dots$ as is common in the theory of relativity.
We use the Einstein summation convention.

For a point~$x$ in $B(\partial M,h)=\{y\in M;d(y,\partial M)<h\}$ we define its projection to~$\partial M$ by setting the normal coordinate to zero: $\bd{x}=(0,x^1,x^2,\dots)$.
We write coordinates as~$x=(x^0,\bd{x})$.

We write~$\bd{\Gamma}$ and~$\bd{\nabla}$ for the Christoffel symbol and the covariant derivative on the submanifold~$\partial M$.
We occasionally write a boundary object (such as~$\bd{\Gamma}$) with Greek indices by extending it by zero: for example, $\bd{\Gamma}^{i}_{\phantom{i}\alpha\beta}=0$ whenever $\alpha=0$ or~$\beta=0$.

We denote derivatives with respect to coordinates by $T_{,\alpha}\coloneqq \partial_\alpha T$ for any function~$T$ on~$M$.
In case of iterated derivatives we do not repeat the comma: $T_{,\alpha\beta}\coloneqq\partial_\beta\partial_\alpha T$.

We define the (scalar) second fundamental form in $B(\partial M,h)$ by
\begin{equation}
\label{eq:sff-def}
\sff{a}{b}=-\frac{1}{2}g_{ij,0}a^ib^j
\end{equation}
for $a,b\in T(B(\partial M,h))$.
This coincides with the usual definition at~$\partial M$ and extends it to the neighborhood $B(\partial M,h)$ in a coordinate independent way.
The splitting of a tangent space if a point in $B(\partial M,h)$ to the tangential and normal spaces of the shortest geodesic to the boundary is independent of the choice of coordinates.
We remark that the normal components~$a^0$ and~$b^0$ do not appear in the second fundamental form $\sff{a}{b}$ even outside~$\partial M$.

We mention the following useful formulas, which are easy to verify:
\begin{equation}
\label{eq:G-facts}
\begin{split}
\Gamma^0_{\phantom{0}ij}&=-\frac{1}{2}g_{ij,0},\\
\Gamma^i_{\phantom{i}00}&=0,\\
\Gamma^i_{\phantom{i}j0}&=\frac{1}{2}g^{ik}g_{kj,0},\\
\Gamma^i_{\phantom{i}jk}&=\bd{\Gamma}^i_{\phantom{i}jk}.
\end{split}
\end{equation}
The last identity is only meaningful on $\partial M$, while the others hold in all of $B(\partial M,h)$.
It follows from the definition of the boundary normal coordinates that $g_{0i}=0$ and~$g_{00}=1$.

For a broken ray~$\gamma$ in $B(\partial M,h)$ we define its \emph{energy} as
\begin{equation}
\E(t)=\frac{1}{2}\dot{\gamma}^0(t)^2+\gamma^0(t)\tpm{\dot{\gamma}(t)}.
\end{equation}
The energy~$\E$ (as well as $\tpm{\dot{\gamma}(t)}$) is well defined and continuous even on points of reflection, since~$\dot{\bd{\gamma}}$ and~$(\dot{\gamma}^0)^2$ are continuous.
As it turns out, $\tpm{\dot{\gamma}(t)}^{-2/3}\E(t)$ has a nicer limit as the broken ray approaches the boundary (see lemma~\ref{lma:rho-conv} and equation~\eqref{eq:rhok}), but we find the energy~$\E$ more convenient to work with.

Let $\sigma:[0,L]\to\partial M$ be a geodesic of finite length on the manifold~$\partial M$. We say that~$\sigma$ is \emph{admissible} if
\begin{itemize}
\item $\sigma(0)\in\bar{E}$ and $\sigma(L)\in\overline{\sisus E}$ (or vice versa) and
\item $\tpm{\dot{\sigma}(t)}>0$ for all $t\in[0,L]$.
\end{itemize}
The geometrical meaning of the second condition is that~$\partial M$ is strictly convex along~$\sigma$.
We will assume that $\sigma(0)\in\bar{E}$ and $\sigma(L)\in\overline{\sisus E}$; the corresponding results for the opposite situation are trivial generalizations of the ones we present.
As we shall show in lemma~\ref{lma:unif-conv}, an admissible boundary geodesic is a~$C^1$ uniform limit of broken rays with endpoints in~$E$.

Given an admissible geodesic~$\sigma$, we construct a sequence of broken rays~$\gamma_n$ as follows.
A broken ray~$\gamma$ is uniquely determined by its initial point and direction; we continue the broken ray for time~$L$.
We let $(\gamma_n(0),\dot{\gamma}_n(0))\to(\sigma(0),\dot{\sigma}(0))$, $\dot{\gamma}^0_n(0)>0$, and $\gamma_n(0)\in E$.
In this limit $\dot{\gamma}_n^0(0)\to0$ as $n\to\infty$.
We denote the energy of~$\gamma_n$ by~$\E_n$; we have assumed thus that $\E_n(0)\to0$.
Any sequence of broken rays satisfying these assumptions converges to the boundary geodesic~$\sigma$ as shown in lemma~\ref{lma:unif-conv} below.

We will denote by~$\gamma$ any broken ray sufficiently closee to a given admissible boundary geodesic~$\sigma$ and by~$\E$ the corresponding energy.
When we refer to the broken rays~$\gamma_n$ in the sequence constructed above, we denote the energy by~$\E_n$.

The constants implied by the notations~$\Order$ and~$\order$ are uniform (for a fixed~$\sigma$) and the estimates hold in the limit $n\to\infty$ without mention.
When otherwise stating that a constant is uniform, we mean that it is independent of~$n$ and~$t\in[0,L]$.

We use the notations $\N=\{0,1,\dots\}$, $\Nb=\N\cup\{\infty\}$, and $\set{k}=\{i\in\N:i\leq k\}$ for $k\in\Nb\cup\{-1\}$.

The normal derivative $\partial_\nu f$ of a function $f:M\to\R$ is well defined in the neighborhood $B(\partial M,h)$ if the limit of the difference quotient exists in the boundary normal coordinates.
We define
\begin{equation}
C_\nu^k=\{f:M\to\R;\partial_\nu^if\text{ is continuous in $B(\partial M,h)$ for all }i\in\set{k}\}
\end{equation}
and $C_\nu^\infty=\bigcap_{k=0}^\infty C_\nu^k$.

We will occasionally omit the time argument in various functions depending on~$t$.

\subsection{Statement of results}
\label{sec:thm}

These results are stated in the notations and assumptions of section~\ref{sec:ass}.
Unless otherwise mentioned, $\gamma$ is a broken ray in $B(\partial M,h)$.

\begin{theorem}
\label{thm:f-rec}
Let~$(M,g)$ be a Riemannian manifold with boundary satisfying the assumptions of section~\ref{sec:ass} and define the weight~$W_\sigma$ by~\eqref{eq:W-def}.
Let $k\in\N$.
For any admissible geodesic $\sigma:[0,L]\to\partial M$ and $f\in C_\nu^k$ one can reconstruct the integral
\begin{equation}
\label{eq:kdf-int}
\int_0^LW_\sigma(t)\tpm{\dot{\sigma}(t)}^{-k/3}\partial_\nu^kf(\sigma(t))\der t
\end{equation}
from the broken ray transform~$\brt{W}f$ (on the whole~$\Gamma_E$) and the knowledge of~$\partial_\nu^if$ for all $i\in\set{k-1}$ in a neighborhood of~$\sigma$ on~$\partial M$.
\end{theorem}

In the case~$k=0$ the theorem simply states that one can reconstruct the weighted integral of~$f$ over any admissible~$\sigma$. 
We also have the following immediate corollary.

\begin{corollary}
\label{cor:f-rec}
Let $k\in\Nb$ and $f\in C_\nu^k$.
Suppose the data set
\begin{equation}
\label{eq:Fk-data}
\left\{\left(\sigma,\int_0^LW_\sigma(t)\tpm{\dot{\sigma}(t)}^{-i/3}F(\sigma(t))\der t\right);\sigma:[0,L]\to\partial M\text{ admissible}\right\}
\end{equation}
determines $F\in C(\partial M,\R)$ uniquely for every~$i\in\set{k}$.
Then one can recover $\partial_\nu^if|_{\partial M}$ for every $i\in\set{k}$ from~$\brt{W}f$.
\end{corollary}

\begin{proof}
By theorem~\ref{thm:f-rec} we may recover the data~\eqref{eq:Fk-data} with~$i=0$ from~$\brt{W}f$.
By assumption, this determines~$f|_{\partial M}$.
Using theorem~\ref{thm:f-rec} again, we may thus recover the data~\eqref{eq:Fk-data} with~$i=1$.
This again determines $\partial_\nu f|_{\partial M}$.
Continuing inductively up to order~$k$ verifies the claim.
\end{proof}

Emergence of the second fundamental form in the weight is due to the fact that a geodesic near the boundary is on average closer to the boundary where the second fundamental form is large.
This phenomenon is illustrated in the ellipse in figure~\ref{fig:ell}.

\begin{figure}%
\includegraphics[width=\columnwidth]{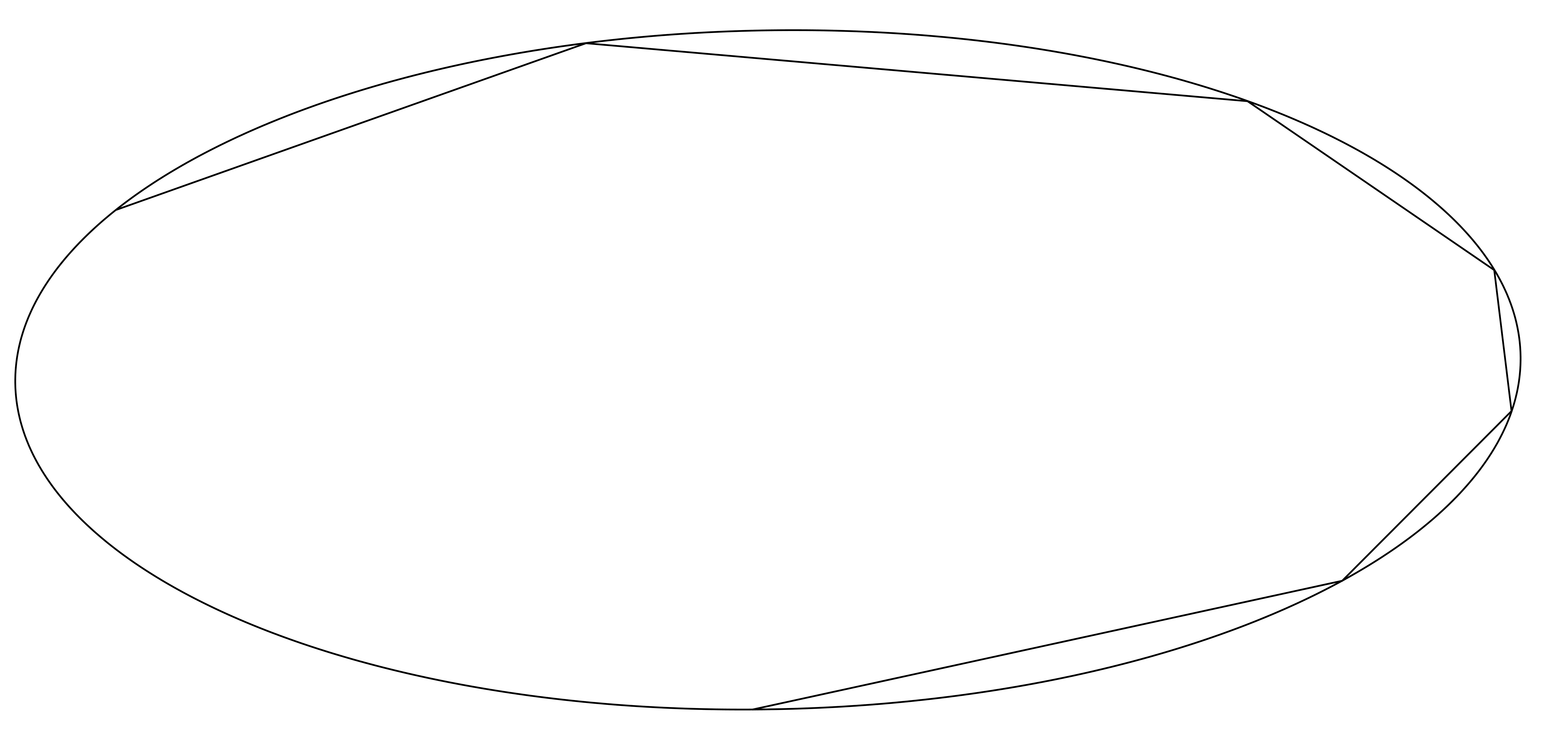}%
\caption{A broken ray near the boundary in an ellipse. The geodesic is closer to the boundary where curvature is larger. This causes the second fundamental form to appear in the weight in theorem~\ref{thm:f-rec}. See text for details.}%
\label{fig:ell}%
\end{figure}

If the data~\eqref{eq:Fk-data} is enough to reconstruct~$F$ on a suitable subset of~$\partial M$, the argument in corollary~\ref{cor:f-rec} can be used to prove reconstruction on this subset.
We present a version of this generalization as the following corollary.
This is in fact a local boundary determination result for the geodesic ray transform.
We do not claim that the result is new, but we present it because it follows easily from theorem~\ref{thm:f-rec}.

\begin{corollary}
\label{cor:f-rec-E}
Let $k\in\Nb$ and $f\in C_\nu^k$.
Suppose that for every $x\in\sisus E$ there is $v\in T_x(\partial M)$ with $\tpm{v}>0$ and $w(x,v)\neq0$.
Then one can recover $\partial_\nu^if|_{\overline{\sisus E}}$ for every $i\in\set{k}$ from~$\brt{W}f$.

In fact, it suffices to know $\brt{W}f(\gamma)$ for geodesics~$\gamma$ (without reflections) with endpoints in~$E$.
\end{corollary}

\begin{proof}
We prove reconstruction in $\sisus E$; the rest follows from continuity.

For $x\in\sisus E$, denote by $\sigma_{x,v}^T$ the unique boundary geodesic $\sigma_{x,v}^T:[0,T]\to\partial M$ with initial position and speed~$(x,v)$.
For~$T$ small enough the geodesic lies in $\sisus E$ and is admissible.
We may now reconstruct the limit
\begin{equation}
\begin{split}
&\lim_{T\to0}\frac{1}{T}
\int_0^TW_{\sigma_{x,v}}(t)\tpm{\dot{\sigma}_{x,v}(t)}^{-i/3}\partial_\nu^{i}f(\sigma_{x,v}(t))\der t
\\&\qquad=
w(x,v)\tpm{v}^{-i/3}\partial_\nu^{i}f(x)
\end{split}
\end{equation}
and therefore also $\partial_\nu^{i}f(x)$ for $i=0$ from $\brt{W}f$ by theorem~\ref{thm:f-rec}.
For~$i>0$ we iterate the argument as in the proof of corollary~\ref{cor:f-rec}.

Let us now turn to the last claim.
The proof of theorem~\ref{thm:f-rec} only uses broken rays near~$\sigma$, whence in the present case it is enough to consider broken rays with all reflections in~$E$.
One can easily construct the broken ray transform for such broken rays from the knowledge of broken ray transform for the geodesic segments.
\end{proof}


\begin{remark}
\label{rmk:local}
The reconstruction scheme is based on broken rays very close to the boundary.
Hence it does not matter if there are obstacles, singularities or any kind of abnormalities in the manifold, as long as there is a neighborhood $B(\partial M,h)$ of the boundary where the assumptions of section~\ref{sec:ass} are met.
The parameter~$h$ can be chosen as small as desired.
\end{remark}

\begin{remark}
\label{rmk:conf}
If the data~\eqref{eq:Fk-data} with~$i=0$ determines~$F$ and the second fundamental form satisfies $\sff{a}{b}=c\ip{a}{b}$ for some nonvanishing function $c:\partial M\to\R$ (i.e. the first and second fundamental forms are conformally equivalent), then the data~\eqref{eq:Fk-data} with any~$i$ determines~$F$.
On a Euclidean sphere~$S^n$, for example, this coefficient~$c$ is a constant.
\end{remark}

As an example of remark~\ref{rmk:conf}, consider the ball $B^n\subset\R^n$ where $n\geq3$ and $E\supset\{x\in S^{n-1};x_1<\eps\}$ for some~$\eps>0$.
Now the first and second fundamental forms at the boundary are conformal and $\{x\in S^{n-1};x_1\geq\eps\}$ is simple, whence one can reconstruct the full Taylor polynomial of a smooth function at the boundary from its broken ray transform; for reconstruction on~$E$ one can apply corollary~\ref{cor:f-rec-E}.
For injectivity of the geodesic ray transform on simple manifolds we refer to~\cite{DKSU:anisotropic}.
Interior reconstruction for the broken ray transform in the Euclidean ball was shown recently~\cite{I:disk} when $n\geq2$ and $E\subset S^{n-1}$ is any open set,
but only by assuming that the unknown function is (in a suitable sense) analytic in the angular variable.

\begin{remark}
\label{rmk:tensor}
Theorem~\ref{thm:f-rec} holds true also when~$f$ is a tensor field in the case~$k=0$.
In fact, it may be any continuous function on the tangent bundle, and proof is the same as for scalars.
Recovery of derivatives and normal components of a tensorial~$f$, however, would be a nontrivial generalization.
\end{remark}

The important question remains whether the data~\eqref{eq:Fk-data} determines the continuous function~$F$.
If~$E$ is an open subset of~$\partial M$, then every boundary geodesic with endpoints in~$\bar{E}$ along which~$\partial M$ is strictly convex is admissible.
In this case the question amounts to asking whether the weighted geodesic ray transform on~$\partial M\setminus E$ (restricted to admissible geodesics) is injective.
The weight contains only attenuation generated by~$\alpha$ if~$i=0$ and~$w\equiv1$.


If not all boundary geodesics are admissible, one ends up with the geodesic ray transform with partial data.
This question was dealt with by Frigyik, Stefanov, and Uhlmann~\cite{FSU:general-x-ray}.
Their result applies in dimension two or higher when the weight is a~$C^2$ perturbation of an analytic one and suitable geometric conditions are met. 
The weights appearing in the data~\eqref{eq:Fk-data} are analytic if the metric~$g$ and the boundary~$\partial M$ are analytic.

In the case of full data the answer is more complete.
In the absence of weight or attenuation injectivity was shown by Mukhometov~\cite{M:bdy-rig-surface-rus} (English translation~\cite{M:bdy-rig-surface-eng}) for some Riemannian surfaces (including simple ones).
Injectivity has later been shown for the attenuated ray transform on simple surfaces~\cite{SU:surface} and simple manifolds of any dimension with small attenuation~\cite[Theorem~7.1]{DKSU:anisotropic}.

In dimension three or higher a local injectivity result was recently obtained~\cite{UV:local-x-ray} and subsequently generalized to the case of arbitrary positive weights~\cite[Corollary~3.2]{SUV:partial-bdy-rig}.
If the manifold admits a suitable convex foliation, these injectivity results become global.

If the dimension~$m-1$ of~$\partial M$ is three or higher, the problem is overdetermined.
If~$m=3$, the problem is formally determined.
The Radon transform of compactly supported functions in the plane is not injective for all smooth positive weights (see e.g.~\cite{B:weighted-radon-noninjective} and references therein), and one cannot expect that an arbitrarily weighted geodesic ray transform would be injective on a Riemannian surface with boundary.

The case~$m=2$ is even worse, since the ray transform in one dimension is not invertible.
However, all metrics on a one dimensional manifold are conformally equivalent, and remark~\ref{rmk:conf} is easy to use.
If the ray transform is known with all constant attenuations, boundary determination is possible, as discussed in section~\ref{sec:1d}.

Injectivity results for the geodesic ray transform on a manifold often assume convexity of some kind.
Therefore our boundary reconstruction method works best when the geometry is doubly strictly convex in the sense that~$M$ has a strictly convex boundary and also $\partial M\setminus E$ (or $\partial M\setminus E'$ for some $E'\subset E$) is strictly convex.

The assumptions of admissibility of geodesics in theorem~\ref{thm:f-rec} cannot be relaxed easily.
As the proof will demonsrate, the assumption on endpoints is necessary for the present reconstruction method.
The second condition states that the boundary is strictly convex along the boundary geodesic.
If the second fundamental form is negative along a boundary geodesic, broken rays starting close to it fail to follow it and escape into the interior of the manifold.

We only assume that the metric~$g$ on~$M$ (and thus on~$\partial M$) is~$C^3$, but we can still recover integrals of $\partial_\nu^kf$ for~$k>3$.
Although surprising, this makes sense; the geodesic normal to the boundary is well defined at any boundary point, and when restricted to this geodesic any function on the manifold becomes a function on some interval on the real line, where smoothness of any order is meaningful.

\section{Convergence to admissible geodesics}
\label{sec:lma}

\subsection{Uniform convergence}

In this section we prove the following lemma and study the necessary details of this convergence.

\begin{lemma}
\label{lma:unif-conv}
For an admissible geodesic $\sigma:[0,L]\to\partial M$ and a sequence $(\gamma_n)_{n\in\N}$ of broken rays satisfying the assumptions of section~\ref{sec:ass}, we have
\begin{equation}
\label{eq:unif-conv}
\begin{split}
&\gamma^0_n\to0,\\
&\dot{\gamma}^0_n\to0,\\
&\bd{\gamma}_n\to\sigma,\\
&\dbd{\gamma}_n\to\dot{\sigma},\text{ and}\\
&\E_n\to0
\end{split}
\end{equation}
uniformly on~$[0,L]$.
\end{lemma}

The proof of this lemma will be given right after lemma~\ref{lma:1st-conv}.

We split the interval~$[0,L]$ into~$N$ parts of length~$L/N$ in such a way that
for each $i\in\{0,\dots,N-1\}$ there is a chart on~$\partial M$ such that the chart contains the ball $B(\sigma(iL/N),2L/N)$. 

We also assume~$N$ to be so large that the following holds for some~$\delta>0$:
if $\alpha:[0,L/N]\to M$ is a~$C^1$ curve with $\abs{\dot{\alpha}}\leq2$, $\abs{\nabla_{\dot{\alpha}}\dot{\alpha}}\leq1$ and
\begin{equation}
d_{TM}((\alpha(0),\dot{\alpha}(0)),(\sigma(t_0),\dot{\sigma}(t_0)))
<
\delta
\end{equation}
for some $t_0\in[0,L(1-1/N)]$, then
\begin{equation}
\label{eq:sff-ell}
\tpm{\dot{\alpha}(t)}
\geq
\frac{1}{2}\tpm{\dot{\sigma}(t_0)}
\end{equation}
for all~$t\in[0,L/N]$.
This choice of~$N$ is possible due to continuity of of~$\tpm{\cdot}$.

\begin{lemma}
\label{lma:chart-conv}
Let $\sigma:[0,T]\to\partial M$ be a geodesic on a single coordinate chart of~$\partial M$ in such a way that the chart contains the ball $B(\sigma(0),2T)$.
Suppose we have a sequence of curves~$(\sigma_n)$ parametrized by~$[0,T]$ on this chart satisfying
\begin{equation}
\begin{split}
&\abs{\dot{\sigma}_n}\to1\text{ uniformly,}
\\
&\abs{\nabla_{\dot{\sigma}_n}\dot{\sigma}_n}\to0\text{ uniformly,}
\\
&\sigma_n(0)\to\sigma(0)\text{, and}
\\
&\dot{\sigma}_n(0)\to\dot{\sigma}(0).
\end{split}
\end{equation}
Then for any~$t\in[0,T]$ we have
\begin{equation}
(\sigma_n(t),\dot{\sigma}_n(t))\to(\sigma(t),\dot{\sigma}(t)),
\end{equation}
and the convergence is uniform.
\end{lemma}

\begin{proof}
The sequence $((\sigma_n(t),\dot{\sigma}_n(t)))_n$ in $\R^{2m}$ is uniformly bounded on the interval~$[0,T]$.
Due to the assumption that $\abs{\dot{\sigma}_n}\to1$ uniformly we eventually have $\abs{\dot{\sigma}_n}<2$.
Thus the curve $\sigma_n([0,T])$ is eventually contained in the chart.

We start by showing that there is a subsequence converging uniformly to~$(\sigma,\dot{\sigma})$.
By the Ascoli-Arzel\`a theorem, there is a uniformly converging subsequence.
By the assumptions, the limit~$\sigma_\infty$ is a geodesic with $\sigma_\infty(0)=\sigma(0)$, and $\dot{\sigma}_\infty(0)=\dot{\sigma}(0)$.
By uniqueness of geodesics,~$\sigma_\infty=\sigma$.

Suppose then that the full sequence does not converge to~$(\sigma,\dot{\sigma})$.
Then there is a neighborhood $U$ of $(\sigma,\dot{\sigma})$ in $\R^{2m}$ and a subsequence which stays outside~$U$.
By the Ascoli-Arzel\`a theorem this subsequence has a uniformly converging subsequence; denote the limit by~$\sigma_\infty$.
As above, one may conclude that~$\sigma_\infty=\sigma$, which is in contradiction with the assumption that the subsequence lies outside a neighborhood of~$\sigma$.
\end{proof}

\begin{lemma}
\label{lma:bd-par}
If~$\gamma$ is a geodesic in $B(\partial M,h)$, then
\begin{equation}
(\bd{\nabla}_{\dbd{\gamma}}\dbd{\gamma})^i
=
-\dot{\gamma}^0\dot{\gamma}^jg^{ik}(0,\bd{\gamma})g_{kj,0}(0,\bd{\gamma})
-\dot{\gamma}^\alpha\dot{\gamma}^\beta\int_0^{\gamma^0}\Gamma^{i}_{\phantom{i}jk,0}(t,\bd{\gamma})\der t
\end{equation}
for each~$i$.
\end{lemma}
\begin{proof}
Since~$\gamma$ is a geodesic, we have
\begin{equation}
({\nabla}_{\dot{\gamma}}\dot{\gamma})^i
=
\ddot{\gamma}^i+\Gamma^i_{\phantom{i}\alpha\beta}\dot{\gamma}^\alpha\dot{\gamma}^\beta
=
0
\end{equation}
for each~$i$.

Using this we find
\begin{equation}
\begin{split}
(\bd{\nabla}_{\dbd{\gamma}}\dbd{\gamma})^i
&=
\ddbd{\gamma}^i+\bd{\Gamma}^i_{\phantom{i}jk}(0,\bd{\gamma})\dbd{\gamma}^j\dbd{\gamma}^k
\\&=
\bd{\Gamma}^i_{\phantom{i}\alpha\beta}(0,\bd{\gamma})\dbd{\gamma}^\alpha\dbd{\gamma}^\beta-\Gamma^i_{\phantom{i}\alpha\beta}(\gamma^0,\bd{\gamma})\dot{\gamma}^\alpha\dot{\gamma}^\beta
\\&=
[\bd{\Gamma}^i_{\phantom{i}\alpha\beta}(0,\bd{\gamma})-\Gamma^i_{\phantom{i}\alpha\beta}(0,\bd{\gamma})]\dbd{\gamma}^\alpha\dbd{\gamma}^\beta
\\&\quad+
\Gamma^i_{\phantom{i}\alpha\beta}(0,\bd{\gamma})[\dbd{\gamma}^\alpha\dbd{\gamma}^\beta-\dot{\gamma}^\alpha\dot{\gamma}^\beta]
\\&\quad+
[\Gamma^i_{\phantom{i}\alpha\beta}(0,\bd{\gamma})-\Gamma^i_{\phantom{i}\alpha\beta}(\gamma^0,\bd{\gamma})]\dot{\gamma}^\alpha\dot{\gamma}^\beta
\\&=
0
-
2\Gamma^i_{\phantom{i}j0}(0,\bd{\gamma})\dot{\gamma}^0\dot{\gamma}^j
-
\dot{\gamma}^\alpha\dot{\gamma}^\beta\int_0^{\gamma^0}\Gamma^{i}_{\phantom{i}jk,0}(t,\bd{\gamma})\der t.
\end{split}
\end{equation}
This is the claimed result.
\end{proof}

It follows from the geodesic equation with~\eqref{eq:sff-def} and~\eqref{eq:G-facts} that apart from points of reflection~$\gamma$ satisfies
\begin{equation}
\label{eq:gdd0}
\ddot{\gamma}^0+\tpm{\dot{\gamma}}=0.
\end{equation}
From this it follows that
\begin{equation}
\label{eq:dE}
\dot{\E}=\gamma^0\Der{t}\tpm{\dot{\gamma}}.
\end{equation}
We define $A_{ijk}=g_{il,0}\Gamma^l_{\phantom{l}kj}-\frac{1}{2}g_{ij,0k}$ and write $A(v)=A_{ijk}v^iv^jv^k$ for short.
A simple calculation shows that\footnote{
We wish to point out that $(g_{ik,0}g^{kl}g_{lj,0}-\frac{1}{2}g_{ij,00})=-\frac{1}{2}g_{ik}(g^{kl}g_{lr,0}g^{rs})_{,0}g_{sj}$ although we do not use this structure.
}
\begin{equation}
\label{eq:dsff1}
\Der{t}\tpm{\dot{\gamma}}
=
A(\dot{\gamma})
+(g_{ik,0}g^{kl}g_{lj,0}-\frac{1}{2}g_{ij,00})\dot{\gamma}^i\dot{\gamma}^j\dot{\gamma}^0
\end{equation}
and
\begin{equation}
\label{eq:dsff2}
\Der{t}\tpm{\dot{\sigma}}=A(\dot{\sigma}).
\end{equation}
Therefore
\begin{equation}
\label{eq:dE-est}
\dot{\E}
\leq C_1\gamma^0
\end{equation}
for some uniform constant~$C_1$.

Consider now the geodesic segment $\sigma|_{[iL/N,(i+1)L/N]}$, first for~$i=0$.
By admissibility of~$\sigma$ there is a constant~$C_2$ such that $\sff{\dot{\sigma}}{\dot{\sigma}}\geq2C_2$ on this segment.
By the estimate~\eqref{eq:sff-ell} and lemma~\ref{lma:bd-par} this implies that
\begin{equation}
\tpm{\dot{\gamma}_n}\geq C_2
\end{equation}
on this interval.
Thus $\E\geq C_2\gamma^0$ on this interval and thus $\dot{\E}\leq (C_1/C_2)\E$.

By Gr\"onwall's inequality we have $\E(t)\leq\E(0)\exp((C_1/C_2)t)$ for all $t\in[0,L/N]$ and so $\E\leq\E(0)\exp(C_1L/C_2N)$ on this interval.
Thus, as $\E_n(0)\to0$, $\E_n\to0$ uniformly on this interval.
But since $\gamma^0_n\leq\E_n/C_2$ and $\abs{\dot{\gamma}_n^0}\leq\sqrt{2\E_n}$, this implies that also $\gamma_n^0\to0$ and $\dot{\gamma}_n^0\to0$ uniformly.
By lemma~\ref{lma:bd-par} this implies that $\bd{\nabla}_{\dbd{\gamma}}\dbd{\gamma}\to0$ uniformly.

Combining these observations with lemma~\ref{lma:chart-conv} we obtain the following.

\begin{lemma}
\label{lma:1st-conv}
On the interval $[0,L/N]$ we have~\eqref{eq:unif-conv} uniformly.
\end{lemma}

We are now ready to finish the proof of lemma~\ref{lma:unif-conv}.

\begin{proof}[Proof of lemma~\ref{lma:unif-conv}]
By lemma~\ref{lma:1st-conv} we have~\eqref{eq:unif-conv} uniformly on $[0,L/N]$.
In particular, it follows that $\gamma_n(L/N)\to\sigma(L/N)$, $\dot{\gamma}_n(L/N)\to\dot{\sigma}(L/N)$ and $\E_n(L/N)\to0$.
This allows us to use lemma~\ref{lma:1st-conv} again on the interval~$[L/N,2L/N]$.
Although the rate of convergence may deteriorate, the same uniform convergence results hold on the combined interval~$[0,2L/N]$.
Iterating this argument, we finally conclude the proof of lemma~\ref{lma:unif-conv}.
\end{proof}

\begin{remark}
\label{rmk:ell}
Since $\gamma_n\to\sigma$ uniformly in~$C^1$ and $\tpm{\dot{\sigma}}$ is bounded from below by admissibility and compactness, we may also assume that $\tpm{\dot{\gamma}_n}$ is bounded from below with a constant independent of time~$t$ or index~$n$.
In particular, this implies that $\gamma_n^0=\Order(\E_n(0))$.
\end{remark}

\subsection{Further estimates}

The results of the previous section combined with lemma~\ref{lma:hoptime} below are enough to prove theorem~\ref{thm:f-rec} for~$k=0$.
For $k\geq1$ we need more tools, and they will be developed in this section.

\begin{lemma}
\label{lma:hoptime}
Let $\tau$ be the time between two adjacent zeroes of~$\gamma^0$. 
If the first of the two zeros is at time~$t_0$, then
\begin{equation}
\abs{\frac{\tau}{2\sqrt{2\E(t_0)}/\tpm{\dot{\gamma}(t_0)}}-1}=\Order(\sqrt{\E(t_0)}).
\end{equation}
\end{lemma}

\begin{proof}
We take~$t_0=0$ for simplicity and denote $v=\dot{\gamma}^0(0)$ and $a=\tpm{\dot{\gamma}(0)}$.
The claim is thus that
\begin{equation}
\label{eq:hopest}
\abs{\frac{\tau}{\tilde{\tau}}-1}=\Order(v),
\end{equation}
where $\tilde{\tau}=2v/a$.
We have $\ddot{\gamma}^0(t)=-\tpm{\dot{\gamma}(t)}$ and by Lipschitz continuity of the second fundamental form
\begin{equation}
\abs{\tpm{\dot{\gamma}(t)}-\tpm{\dot{\gamma}(0)}}\leq Ct
\end{equation}
for $t\in[0,\tau]$ for some constant~$C$.
By remark~\ref{rmk:ell}~$a$ is bounded from above and below, and the constant~$C$ may also be taken uniform once a boundary geodesic~$\sigma$ is fixed.

The fundamental theorem of calculus gives
\begin{equation}
\begin{split}
\gamma^0(t)
&=
\dot{\gamma}^0(0)t+
\int_0^t\int_0^s\ddot{\gamma}^0(u)\der u\der s
\\&\leq
vt+
\int_0^t\int_0^s(-a+Cu)\der u\der s
\\&=
vt-\frac{1}{2}at^2+\frac{C}{6}t^3
\eqqcolon
h^+(t)
.
\end{split}
\end{equation}
Suppose $v<\frac{3a^2}{8C}$ and denote $\eps=\frac{8C}{3a^2}v\in(0,1)$.
Then $\gamma^0(\tilde{\tau}(1+\eps))\leq h^+(\tilde{\tau}(1+\eps))<0$ and thus $\tau<\tilde{\tau}(1+\eps)$.

Similarly we have
\begin{equation}
\gamma^0(t)
\geq
vt-\frac{1}{2}at^2-\frac{C}{6}t^3
\eqqcolon
h^-(t).
\end{equation}
With the same choice for~$\eps$ and~$v$ we have $\gamma^0(\tilde{\tau}(1-\eps))\geq h^-(\tilde{\tau}(1-\eps))>0$ and thus $\tau>\tilde{\tau}(1-\eps)$.

Combining these results we have
\begin{equation}
\abs{\frac{\tau}{\tilde{\tau}}-1}
<
\eps
=
\frac{8C}{3a^2}v,
\end{equation}
which is the estimate~\eqref{eq:hopest}.
\end{proof}

\begin{remark}
\label{rmk:C2}
The above lemma implies that the points of reflection on a billiard trajectory cannot accumulate.
If we reduce the regularity assumption of the boundary~$\partial M$ from~$C^3$ to $C^2$, there is an example by Halpern~\cite{H:strange-billiard} in the plane where the points indeed do accumulate.
Halpern also proves that~$C^3$ regularity is enough to prevent accumulation~\cite[Theorem~3]{H:strange-billiard}.
\end{remark}

We have the following important heuristics:
\begin{equation}
\label{eq:heur}
\gamma^0\approx\frac{2}{3}\E/\tpm{\dot{\gamma}}
\quad\text{``on average''.}
\end{equation}
Compare this with throwing a ball (under ideal circumstances), when the height at time~$t$ is $h(t)=vt-\frac{1}{2}at^2$, where~$v$ is the initial vertical speed and~$a$ is the gravitational acceleration.
The energy $\E=\frac{1}{2}\dot{h}^2+ah$ is conserved, and the maximum height is~$\E/a$.
It is easy to calculate that the average height on the time interval $[0,2v/a]$ (when the ball is airborne) equals two thirds of the maximum height~$\E/a$.
The role of the gravitational acceleration is played here by the second fundamental form.
We make this heuristic explicit in the following two lemmas.

\begin{lemma}
\label{lma:heur-loc}
Let~$t_0$ and $t_0+\tau$ be two adjacent zeros of~$\gamma^0$. 
Then
\begin{equation}
\int_{t_0}^{t_0+\tau}\gamma^0(t)\der t=\frac{(\dot{\gamma}^0(t_0))^2}{3\tpm{\dot{\gamma}(t_0)}}\tau+\Order(\tau^4).
\end{equation}
\end{lemma}

\begin{proof}
To simplify notation, we assume~$t_0=0$.

We denote $v=\dot{\gamma}^0(0)$, $a=\tpm{\dot{\gamma}(0)}$ and $h(t)=vt-\frac{1}{2}at^2$.
Since $\ddot{\gamma}^0(t)=-\tpm{\dot{\gamma}(t)}$ by~\eqref{eq:gdd0}, the function~$h$ is actually the second order Taylor polynomial of~$\gamma^0$ at~$t=0$.
And because $\ddot{\gamma}^0(t)$ is Lipschitz continuous in~$t$ by~\eqref{eq:gdd0}, we have $\gamma^0(t)-h(t)=\Order(t^3)$.

Elementary calculations show that
\begin{equation}
\int_0^{2v/a}h(t)\der t=\frac{(\dot{\gamma}^0(t_0))^2}{3\tpm{\dot{\gamma}(t_0)}}\tau
\end{equation}
and $h=\Order(\E(0))$.
Lemma~\ref{lma:hoptime} gives $\abs{\tau-\frac{2v}{a}}=\Order(\E(0))=\Order(\tau^2)$, so that
\begin{equation}
\begin{split}
&\int_0^\tau\gamma^0(t)\der t-\int_0^{2v/a}h(t)\der t
\\&\qquad=
\int_0^\tau(\gamma^0(t)-h(t))\der t+\Order\left(\E(0)\abs{\tau-\tfrac{2v}{a}}\right)
\\&\qquad=
\int_0^\tau\Order(t^3)\der t+\Order(\tau^2\cdot\tau^2)
\\&\qquad=
\Order(\tau^4),
\end{split}
\end{equation}
which is the desired estimate.
\end{proof}

\begin{lemma}
\label{lma:heur-glob}
Let~$T>0$ be a zero of~$\gamma^0$ and let $F:[0,T]\to\R$ be a continuous function. 
Then
\begin{equation}
\int_0^T\gamma^0(t)^kF(t)\der t
=
\int_0^T\left(\frac{2\E(t)}{3\tpm{\dot{\gamma}(t)}}\right)^kF(t)\der t+\order(\E(0)^k)
\end{equation}
for any $k\in\N$.
\end{lemma}

\begin{proof}
The case~$k=0$ is trivial.
We give the proof for $k=1$; for larger $k$ the claim follows by iterating the result.

Let $0=t_0<t_1<\dots<t_N=T$ be the zeroes of $\gamma^0$ and denote $\tau=\max_{1\leq i\leq N}(t_i-t_{i-1})$.
By lemma~\ref{lma:hoptime} and uniform convergence of $\E$ to zero (lemma~\ref{lma:unif-conv}) $\Order(\sqrt{\E(0)})$ is the same as~$\Order(\tau)$.
By lemma~\ref{lma:heur-loc} we have
\begin{equation}
\int_{t_{i-1}}^{t_i}\gamma^0(t)\der t
=
\frac{2\E(t_{i-1})}{3\tpm{\dot{\gamma}(t_{i-1})}}(t_i-t_{i-1})+\Order(\tau^4)
\end{equation}
for all $i=1,\dots,N$.
We also have $\gamma^0=\Order(\tau^2)=\Order(\E(0))$.
Thus
\begin{equation}
\begin{split}
&\int_0^T\gamma^0(t)F(t)\der t
\\&\qquad=
\sum_{i=1}^N\int_{t_{i-1}}^{t_i}\gamma^0(t)F(t)\der t
\\&\qquad=
\sum_{i=1}^N\left(\int_{t_{i-1}}^{t_i}\gamma^0(t)F(t_{i-1})\der t+\order(\tau\cdot\tau^2)\right)
\\&\qquad=
\sum_{i=1}^N\left(\frac{2\E(t_{i-1})}{3\tpm{\dot{\gamma}(t_{i-1})}}(t_i-t_{i-1})F(t_{i-1})+\order(\tau^3)\right)
\\&\qquad=
\int_0^T\frac{2\E(t)}{3\tpm{\dot{\gamma}(t)}}F(t)\der t+\order(\tau^2)
\\&\qquad=
\int_0^T\frac{2\E(t)}{3\tpm{\dot{\gamma}(t)}}F(t)\der t+\order(\E(0)),
\end{split}
\end{equation}
since the continuous function~$F$ is Riemann integrable.
\end{proof}

We define $\rho:[0,L]\to(0,\infty)$ by
\begin{equation}
\label{eq:rho-def}
\rho(t)=\exp\left(\int_0^t \frac{2A(\dot{\sigma}(s))}{3\tpm{\dot{\sigma}(s)}}\der s\right).
\end{equation}
This function obviously satisfies
\begin{equation}
\rho'(t)=\rho(t)\frac{2A(\dot{\sigma}(t))}{3\tpm{\dot{\sigma}(t)}}.
\end{equation}
Comparing with the equation $\dot{\E}=\gamma^0A(\dot{\gamma})+\Order(\gamma^0\dot{\gamma}^0)$ (cf.~\eqref{eq:dE} and~\eqref{eq:dsff1}) and the heuristics~\eqref{eq:heur}, we suspect that $\rho_n(t)\coloneqq\E_n(t)/\E_n(0)$ converges to $\rho$ as~$n\to\infty$.

To clarify further notation, we define
\begin{equation}
\phi_n(t)=\frac{2A(\dot{\gamma}_n(t))}{3\tpm{\dot{\gamma}_n(t)}}
\end{equation}
and
\begin{equation}
\rho^n(t)=\exp\left(\int_0^t \phi_n(s)\der s\right).
\end{equation}

\begin{lemma}
\label{lma:rho-conv}
As $n\to\infty$, we have $\rho_n(t)=\rho^n(t)+\order(1)=\rho(t)+\order(1)$ for all~$t\in[0,L]$.
\end{lemma}

\begin{proof}
Since $\gamma_n\to\sigma$ uniformly, it is clear that $\rho^n\to\rho$ uniformly.

It suffices to show that for any $\eps>0$ there is $N\in\N$ such that
\begin{equation}
\label{eq:rhorhoeps}
\abs{\rho_n(t)-\rho^n(t)}\leq\eps
\end{equation}
for all $t\in[0,L]$. 

Let~$t_0$ and $t_0+\tau$ be two adjacent zeros of $\gamma^0$ on~$[0,L]$. 
On this interval, we have by~\eqref{eq:dE} and~\eqref{eq:dsff1}
\begin{equation}
\begin{split}
\E(t_0+\tau)-\E(t_0)
&=
\int_{t_0}^{t_0+\tau}\dot{\E}(t)\der t
\\&=
\int_{t_0}^{t_0+\tau}\gamma^0(t)[A(\dot{\gamma}(t))+\Order(\dot{\gamma}^0(t))]\der t
\\&=
\int_{t_0}^{t_0+\tau}\gamma^0(t)[A(\dot{\gamma}(t_0))+\Order(\tau)+\Order(\dot{\gamma}^0(t))]\der t.
\end{split}
\end{equation}
Since $\gamma^0=\Order(\E)$ (remark~\ref{rmk:ell}), $\dot{\gamma}^0\leq\sqrt{2\E}$, and $\tau=\Order(\sqrt{\E})$ (lemma~\ref{lma:hoptime}), lemma~\ref{lma:heur-loc} gives, after putting $\gamma=\gamma_n$ and dividing by~$\E_n(0)$,
\begin{equation}
\rho_n(t_0+\tau)-\rho_n(t_0)
=
\phi_n(t_0)\rho_n(t_0)\tau+\Order(\E_n(0)).
\end{equation}
But also (recall that $\sqrt{\E}=\Order(\tau)$ by lemma~\ref{lma:hoptime})
\begin{equation}
\rho^n(t_0+\tau)-\rho^n(t_0)
=
\phi_n(t_0)\rho^n(t_0)\tau+\Order(\E_n(0)).
\end{equation}
Defining $\Delta_n(t)=\rho_n(t)-\rho^n(t)$, these estimates give
\begin{equation}
\label{eq:D-est}
\Delta_n(t_0+\tau)=(1+\phi_n(t_0)\tau)\Delta_n(t_0)+\Order(\E_n(0)).
\end{equation}
If we can show that $\Delta_n\to0$ uniformly, we have the estimate~\eqref{eq:rhorhoeps}.

Let $0=t_n^1<t_n^2<\dots<t_n^{N_n}$ be the zeros of $\gamma_n^0$ on~$[0,L]$.
By lemma~\ref{lma:hoptime}, remark~\ref{rmk:ell} and the uniform convergence of $\E$ to zero we have $C_1^{-1}\sqrt{\E_n(0)}\leq t_n^m-t_n^{m-1}\leq C_1\sqrt{\E_n(0)}$ for some $C_1>1$.
This estimate implies that $N_n\leq C_1L/\sqrt{E_n}$.

There is obviously a constant $\Phi>0$ such that $\abs{\phi_n}\leq\Phi$.
We denote the constant associated with $\Order(\E(0))$ in the estimate~\eqref{eq:D-est} by~$C_2$.

From~\eqref{eq:D-est} we obtain
\begin{equation}
\abs{\Delta_n(t_n^l)}\leq C_2\E_n(0)(1+(1+\Phi C_1\sqrt{\E_n(0)})\abs{\Delta_n(t_n^{l-1})})
\end{equation}
for all~$l>1$.
Using~$\Delta_n(0)=0$ and induction turns this into
\begin{equation}
\begin{split}
\abs{\Delta_n(t_n^l)}
&\leq
C_2\E_n(0)\sum_{j=0}^{l-1}(1+\Phi C_1\sqrt{\E_n(0)})^j
\\&=
\frac{C_2\sqrt{\E_n(0)}}{\Phi C_1}[(1+\Phi C_1\sqrt{\E_n(0)})^l-1]
\\&\leq
\frac{C_2\sqrt{\E_n(0)}}{\Phi C_1}[(1+\Phi C_1\sqrt{\E_n(0)})^{C_1L/\sqrt{\E_n(0)}}-1].
\end{split}
\end{equation}
The estimate $(1+y/x)^{zx}<e^{yz}$ for positive~$x,y,z$ yields then
\begin{equation}
\abs{\Delta_n(t_n^l)}
\leq
\frac{C_2\sqrt{\E_n(0)}}{\Phi C_1}[\exp(C_1^2L\Phi)-1]
=
\Order(\sqrt{\E_n(0)}).
\end{equation}
This shows that $\Delta_n\to0$ uniformly and concludes the proof.
%
\end{proof}

By~\eqref{eq:dsff2} we have, in fact,
\begin{equation}
\label{eq:rhok}
\begin{split}
\rho(t)^k
&=
\exp\left(k\int_0^t \frac{2A(\dot{\sigma}(s))}{3\tpm{\dot{\sigma}(s)}}\der s\right)
\\&=
\exp\left(\frac{2k}{3}\int_0^t\Der{s}\log(\tpm{\dot{\sigma}(s)})\der s\right)
\\&=
\left(\frac{\tpm{\dot{\sigma}(t)}}{\tpm{\dot{\sigma}(0)}}\right)^{2k/3}.
\end{split}
\end{equation}

In the reconstruction we will need the Taylor polynomials of~$f$ along geodesics normal to~$\partial M$.
For $p\in\partial M$ which we think of as $p=(0,\bd{x})$, we denote
\begin{equation}
T^k_p f(x^0)=\sum_{i=0}^k \frac{(x^0)^i}{i!}\partial_\nu^if(p).
\end{equation}
If $f\in C_\nu^k$, then
\begin{equation}
\begin{split}
&f(x^0,\bd{x})-T^k_\bd{x} f(x^0)
\\&\qquad=
\idotsint\limits_{\mathclap{0\leq s_k\leq s_{k-1}\leq\dots\leq s_1\leq x^0}}
[\partial_\nu^k f(s_k,\bd{x})-\partial_\nu^k f(0,\bd{x})]\der s_k\cdots\der s_1
\\&\qquad=
\order((x^0)^k)
\end{split}
\end{equation}
and thus
\begin{equation}
\label{eq:df-est}
f(x^0,\bd{x})-T^{k-1}_\bd{x} f(x^0)
=
\frac{(x^0)^k}{k!}\partial_\nu^kf(0,\bd{x})+\order((x^0)^k).
\end{equation}

\section{Proof of theorem~\ref{thm:f-rec}}
\label{sec:pf}

Now we are ready to prove theorem~\ref{thm:f-rec}.
We will assume that $\sigma(L)\in\sisus E$; for endpoints in the closure, there is a sequence of admissible geodesics converging to the desired admissible geodesic, giving the integral~\eqref{eq:kdf-int} for all admissible~$\sigma$.

Let us begin with~$k=0$.
Define
\begin{equation}
L_n=\max\{t\in[0,L];\gamma_n^0(t)=0\}.
\end{equation}
By lemma~\ref{lma:hoptime} $L_n\to L$ as $n\to\infty$.
Since $\sigma(L)\in\sisus E$, we have $\gamma_n(L_n)\in E$ for sufficiently large $n$ (and $\gamma_n(0)\in E$ by assumption).
This is the reason for the first condition in the definition of an admissible geodesic.

The integrals
\begin{equation}
\label{eq:int-known1}
\brt{W}f(\gamma_n|_{[0,L_n]})
=
\int_0^{L_n}W_{\gamma_n}(t)f(\gamma_n(t))\der t
\end{equation}
are known by assumption, and as $n\to\infty$, this approaches the integral~\eqref{eq:kdf-int} by lemma~\ref{lma:unif-conv}.
This proves the theorem for~$k=0$.

Let now $k\geq1$.
Again, the integral~\eqref{eq:int-known1} is known.
Also, since normal derivatives of~$f$ with order strictly less than~$k$ are known on each~$\gamma_n$ for sufficiently large~$n$ by assumption, by~\eqref{eq:df-est} the integrals
\begin{equation}
\label{eq:int-known2}
\int_0^{L_n}W_{\gamma_n}(t)T^{k-1}_{\bd{\gamma}_n(t)} f(\gamma_n^0(t))\der t
\end{equation}
can be calculated.

Thus, if we show that
\begin{equation}
\label{eq:pf}
\begin{split}
&
k!\left(\frac{3\tpm{\dot{\gamma}_n(0)}^{2/3}}{2\E_n(0)}\right)^k
\\&\quad\times
\int_0^{L_n}W_{\gamma_n}(t)[f(\gamma_n(t))-T^{k-1}_{\bd{\gamma}_n(t)} f(\gamma_n^0(t))]\der t
\\&=
\int_0^LW_\sigma(t)\tpm{\dot{\sigma}(t)}^{-k/3}\partial_\nu^k f(\sigma(t))\der t+\order(1),
\end{split}
\end{equation}
we obtain the integral~\eqref{eq:kdf-int} from the given data.

By~\eqref{eq:df-est}, the estimate $\gamma_n^0=\Order(\E_n(0))$ (see remark~\ref{rmk:ell}), and lemma~\ref{lma:heur-glob} with $F(t)=W_{\gamma_n}(t)\partial_\nu^k f(\bd{\gamma}_n(t))$ we get
\begin{equation}
\begin{split}
&
k!\int_0^{L_n}W_{\gamma_n}(t)[f(\gamma_n(t))-T^{k-1}_{\bd{\gamma}_n(t)} f(\gamma_n^0(t))]\der t
\\&=
\int_0^{L_n}W_{\gamma_n}(t)[\gamma_n^0(t)^k\partial_\nu^k f(\bd{\gamma}_n(t))+\order(\gamma_n^0(t)^k)]\der t
\\&=
\int_0^{L_n}W_{\gamma_n}(t)\gamma_n^0(t)^k\partial_\nu^k f(\bd{\gamma}_n(t))\der t+\order(\E_n(0)^k)
\\&=
\int_0^{L_n}W_{\gamma_n}(t)\left(\frac{2\E_n(t)}{3\tpm{\dot{\gamma}_n(t)}}\right)^k\partial_\nu^k f(\bd{\gamma}_n(t))\der t+\order(\E_n(0)^k).
\end{split}
\end{equation}
Thus, after simplification, we have
\begin{equation}
\begin{split}
&
k!\left(\frac{3\tpm{\dot{\gamma}_n(0)}^{2/3}}{2\E_n(0)}\right)^k
\\&\quad\times
\int_0^{L_n}W_{\gamma_n}(t)[f(\gamma_n(t))-T^{k-1}_{\bd{\gamma}_n(t)} f(\gamma_n^0(t))]\der t
\\&=
\tpm{\dot{\gamma}_n(0)}^{2k/3}
\\&\quad\times
\int_0^{L_n}W_{\gamma_n}(t)\left(\frac{\rho_n(t)}{\tpm{\dot{\gamma}_n(t)}}\right)^k\partial_\nu^k f(\bd{\gamma}_n(t))\der t+\order(1).
\end{split}
\end{equation}
By lemmas~\ref{lma:rho-conv} and~\ref{lma:unif-conv} $\rho_n\to\rho$ and $(\gamma_n,\dot{\gamma}_n)\to(\sigma,\dot{\sigma})$ uniformly.
Combining these with the identity \eqref{eq:rhok} and the result $L_n\to L$ proves the estimate~\eqref{eq:pf} and concludes the proof of theorem~\ref{thm:f-rec}.

\section{One dimensional boundaries}
\label{sec:1d}

We now turn to the case when~$M$ is a compact surface with boundary, and assume the surface to have a strictly convex boundary.
By theorem~\ref{thm:f-rec} the integral~\eqref{eq:kdf-int} can be recovered for~$k=0$ if the broken ray transform of an unknown function~$f$ is known.
This information is clearly not enough to determine~$f$, but if the broken ray transform is known for several weights~$W$, reconstruction may be possible.
It is also worth mentioning that at the boundary of a strictly convex surface the first and second fundamental forms are conformally equivalent, so remark~\ref{rmk:conf} can be used for recovery of normal derivatives.

We formulate the following result to illustrate this idea.

\begin{proposition}
\label{prop:1d-lap}
Let $M$ be a $C^3$ compact surface with strictly convex boundary, let $E\subset\partial M$ be open, $k\in\Nb$, and $\lambda_0\in\R$.
Assume that the broken ray transform of a function $f\in C^k_\nu$ vanishes for all weights~$W$ with $w\equiv0$ and $a\equiv-\lambda$ for $\lambda>\lambda_0$.
Then $\partial_\nu^i f=0$ everywhere at $\partial M$ for each~$i\in\set{k}$.
\end{proposition}

\begin{proof}
By corollary~\ref{cor:f-rec-E} the claim is true at~$E$.
We assume that the rest of the boundary $R=\partial M\setminus E$ is connected; if it is not, the argument may be used on each connected component separately.
We give the proof only for $k=0$; the rest follows from remark~\ref{rmk:conf} and an induction similar to the one in the proof of corollary~\ref{cor:f-rec}.

Using arc length parametrization we write $R=[0,L]$ for some~$L>0$.
(This requires that~$R$ is not only a point, but the single point case is trivial.)
We interpret the restriction~$f|_R$ as a function on $[0,L]$ and denote it by~$f$.
By theorem~\ref{thm:f-rec} and the assumption
\begin{equation}
\int_0^L e^{-\lambda s}f(s)\der s=0
\end{equation}
for each~$\lambda>\lambda_0$.

After extending $f$ by zero to $[0,\infty)$, this means that the Laplace transform~$\lap f(\lambda)$ of~$f$ vanishes for all $\lambda>\lambda_0$.
The Laplace transform of a compactly supported function is analytic, so this implies that~$\lap f(\lambda)=0$ for all~$\lambda\in\R$.
Injectivity of the Laplace transform gives the claim.
\end{proof}

In a similar vein, we study boundary recovery of the potential in a Schr\"odinger equation from the partial Dirichlet to Neumann map.
This problem can in some cases be reduced to boundary reconstruction for the broken ray transform in the following sense~\cite{KS:calderon}:
Suppose for a compactly supported continuous function $f:\R\times M\to\R$ the integral
\begin{equation}
\int_0^L\int_\R e^{-2\lambda(t+ix)}f(x,\gamma(t))\der x\der t
\end{equation}
is known for each broken ray $\gamma:[0,L]\to M$.\footnote{Starting with partial Cauchy data, one ends up with this integral by constructing solutions to the Schr\"odinger equation concentrating near a given broken ray.
The argument involves complex geometrical optics solutions based on reflected Gaussian beam quasimodes.
For a more detailed description, see~\cite{KS:calderon}.}
Can we recover~$f|_{\partial M}$?
The function~$f$ is essentially the difference of two potentials for the Schr\"odinger equation, which one attempts to observe from the Cauchy data.

Theorem~\ref{thm:f-rec} and parametrization of the reflecting part of the boundary by~$[0,L]$ as in the proof of proposition~\ref{prop:1d-lap} reduce the problem to asking whether
\begin{equation}
\int_0^L\int_\R e^{-2\lambda(t+ix)}f(x,t)\der x\der t=0
\end{equation}
for all $\lambda\in\C$ (the expression is complex analytic in~$\lambda$, so vanishing for all real values implies the same for all complex values) implies that~$f\equiv0$, where~$f$ is interpreted as a function on $\R\times[0,L]$.

Denoting $z=t+ix$ and extending~$f$ by zero to~$\C$ turns the condition into
\begin{equation}
\int_\C e^{-2\lambda z}f(z)\der\h^2(z)=0
\end{equation}
for all $\lambda\in\C$, when~$f$ is a compactly supported function in~$\C$.
Here~$\h^2$ is the two dimensional Hausdorff (or Lebesgue) measure.

To study this, we define the integral transform $I:L^1_0\to C$, where~$L^1_0$ is the space of integrable, compactly supported, measurable complex functions in the complex plane and~$C$ is the space of complex analytic functions, by setting
\begin{equation}
If(\lambda)=\int_\C e^{\lambda z}f(z)\der\h^2(z).
\end{equation}
We define scaling and translation in $L^1_0$ by $S_\mu f(z)=f(\mu z)$ for $\mu>0$ and $T_wf(z)=f(z-w)$ for~$w\in\C$.
Simple calculations show that
\begin{equation}
\label{eq:I-sc-tr}
\begin{split}
I(S_\mu f)(\lambda)&=\mu^{-2}If(\lambda/\mu)
\quad\text{and}\\
I(T_w f)(\lambda)&=e^{\lambda w}If(\lambda).
\end{split}
\end{equation}
For the convolution of two functions $f,g\in L^1_0$ (which is still in~$L^1_0$) we have
\begin{equation}
\label{eq:I-conv}
I(f*g)(\lambda)=\int_\C e^{\lambda w}f(w)Ig(\lambda)\der\h^2(w).
\end{equation}

If~$f$ is rotationally symmetric (i.e. $f(e^{i\theta}z)=f(z)$ for all $z\in\C$ and $\theta\in\R$), elementary calculations with a symmetry argument show that all derivatives of~$If$ with respect to~$\lambda$ vanish at~$\lambda=0$.
By analyticity, this implies that~$If$ is constant.
For such~$f$ thus $If(\lambda)=If(0)=\int_\C f\der\h^2$ for all~$\lambda\in\C$.

Therefore the kernel of~$I$ contains (but may not be limited to) functions rotationally symmetric with respect to any point in the plane that integrate to zero,
and convolutions of such functions against functions in~$L^1_0$.
However, $\int_\C f\der\h^2=0$ does not generally imply that~$If=0$, as the example of subtraction of characteristic functions of two squares shows:
if
\begin{equation}
f(a+bi)
=
\begin{cases}
1 & \text{if }a,b\in[0,1]\\
-1 & \text{if }a,b\in[-1,0]\\
0 & \text{otherwise},
\end{cases}
\end{equation}
we have $\int_\C f\der\h^2=0$ but~$If\neq0$.

In conclusion, some information of~$f$ is contained in~$If$, but some is also lost.
Translating this to the original partial data problem, the reconstruction method based on theorem~\ref{thm:f-rec} cannot fully detect the unknown potential, but can detect some properties of it.

\section*{Acknowledgements}
The author is partly supported by the Academy of Finland (no 250~215).
The author wishes to thank Mikko Salo for discussions regarding this article and the referee for useful feedback that has made this article an easier read.
Part of the work was done during a visit to the Institut Mittag-Leffler (Djursholm, Sweden).

\bibliographystyle{abbrv}
\bibliography{disk}

\end{document}